\documentclass{amsart}
\usepackage{amsmath, amsthm, amssymb, amscd, epsfig}

\begin{document}

\newtheorem{theorem}{Theorem}
\newtheorem{lemma}[theorem]{Lemma}
\newtheorem*{proposition}{Proposition}
\newtheorem{corollary}[theorem]{Corollary}
\newtheorem{conjecture}[theorem]{Conjecture}
\newtheorem{question}[theorem]{Question}
\newtheorem{problem}[theorem]{Problem}
\newtheorem*{claim}{Claim}
\newtheorem*{criterion}{Criterion}
\newtheorem*{stability_thm}{Theorem~\ref{stability_theorem}}

\theoremstyle{definition}
\newtheorem{definition}[theorem]{Definition}
\newtheorem{construction}[theorem]{Construction}
\newtheorem{notation}[theorem]{Notation}

\theoremstyle{remark}
\newtheorem{remark}[theorem]{Remark}
\newtheorem{example}[theorem]{Example}

\numberwithin{equation}{subsection}

\def\Z{\mathbb Z}
\def\R{\mathbb R}
\def\C{\mathbb C}
\def\H{\mathbb H}

\def\scl{\textnormal{scl}}

\def\Id{\textnormal{Id}}
\def\PSL{\textnormal{PSL}}
\def\SL{\textnormal{SL}}
\def\til{\widetilde}

\title{Certifying incompressibility of non-injective surfaces with scl}
\author{Danny Calegari}
\address{DPMMS \\ University of Cambridge \\
Cambridge CB3 0WA England}
\email{dcc43@dpmms.cam.ac.uk}
\date{\today}

\begin{abstract}
Cooper--Manning \cite{Cooper_Manning} and Louder \cite{Louder}
gave examples of maps of surface groups to $\PSL(2,\C)$ 
which are not injective, but are incompressible (i.e.\/ no simple
loop is in the kernel).
We construct more examples with very simple {\em certificates} for
their incompressibility arising from the theory of stable commutator length.
\end{abstract}

\maketitle

The purpose of this note is to give examples of maps of closed
surface groups to $\PSL(2,\C)$
which are not $\pi_1$-injective, but are geometrically incompressible, 
in the sense that no simple loop in the surface is in the kernel (in the 
sequel we use the word ``incompressible'' as shorthand for ``geometrically
incompressible''). The examples are very explicit, and the
images can be taken to be all loxodromic.
The significance of such examples is that they
shed light on the simple loop conjecture, which says that any non-injective
map from a closed oriented surface to a $3$-manifold should be compressible.

Examples of such maps were first shown to exist by Cooper--Manning 
\cite{Cooper_Manning}, by a representation variety argument,
thereby answering a question of Minsky \cite{Minsky}
(also see Bowditch \cite{Bowditch}).
More sophisticated examples were then found by Louder \cite{Louder}; 
he even found examples with the property that the minimal self-crossing
number of a loop in the kernel can be taken to be arbitrarily large.
Louder's strategy is to exhibit an explicit finitely presented group
(a limit group) which admits non-injective incompressible surface maps,
and then to observe that such a group can be embedded as an
all-loxodromic subgroup of $\PSL(2,\C)$.

It is easy to produce examples of non-injective surface groups. What is
hard is to certify that they are incompressible. The main point of
our construction, and the main novelty and interest of this paper,
is to show that stable commutator length (and its cousin Gromov--Thurston norm) 
can be used to certify incompressibility.

Our examples are closely related to Louder's examples, although our
certificates are quite different. So another purpose of this note is to
advertise the use of stable commutator length as a tool to get at the
kind of information that is relevant in certain contexts in the theory
of limit groups.

We move back and forward between (fundamental) groups and spaces in the 
usual way. We assume the reader is familiar with stable commutator
length, and Gromov--Thurston norms in dimension 2. Standard references
are \cite{Calegari_scl,Gromov,Thurston}. Computations are done with
the program {\tt scallop}, available from \cite{scallop}.

Recall that if $X$ is a $K(\pi,1)$, the {\em Gromov--Thurston norm} of
a class $\alpha \in H_2(X;\Z)$ (denoted $\|\alpha\|$) 
is the infimum of $-\chi(T)/n$ over
all closed, oriented surfaces $T$ without spherical components mapping to $X$
and representing $n\alpha$. Our certificates for
incompressibility are guaranteed by the following Proposition.

\begin{proposition}[Certificate]\label{certificate_prop}
Let $X$ be a $K(\pi,1)$, and let $\alpha \in H_2(X;\Z)$ be represented
by a closed oriented surface $S$ with no torus or spherical components. If there is
a strict inequality $\|\alpha\| > -\chi(S)-2$ (where $\|\cdot\|$ denotes
Gromov--Thurston norm) then $S$ is (geometrically) incompressible.
\end{proposition}
\begin{proof}
If $S$ is compressible, then $\alpha$ is represented by the result of
compressing $S$, which is a surface $S'$ with no spherical components,
and $-\chi(S') < \|\alpha\|$. But this contradicts the definition of
$\|\alpha\|$.
\end{proof}

On the other hand, a closed surface $S$ without torus or spherical components
representing $\alpha$
and with $-\chi(S) = \|\alpha\|$ {\em is} $\pi_1$-injective, so to apply
our proposition to obtain examples, we must find examples of spaces $X$
and integral homology classes $\alpha$ for which $\|\alpha\|$ is not equal
to $-\chi(S)$ for any closed orientable surface $S$; i.e.\/ for which $\|\alpha\|$
is not in $2\Z$.
Such spaces can never be $3$-manifolds, by combined results of
Gabai and Thurston \cite{Gabai,Thurston}, so our methods will never
directly find a counterexample to the simple loop conjecture.

The groups we consider are all obtained by amalgamating two simpler groups
over a cyclic subgroup. The generator of the cyclic group is homologically
trivial in either factor, giving rise to a class in $H_2$ in the big group.
The Gromov--Thurston norm of this class is related to the stable commutator
length of the loop in the two factors as follows:

\begin{proposition}[Amalgamation]
Let $G$ be an amalgamated product $G=J*_{\langle w \rangle} K$ 
along a cyclic group $\langle w \rangle$
which is generated by a loop $w$ which is homologically trivial on either side.
Let $\phi: H_2(G;\Z) \to H_1(\langle w \rangle;\Z)$ be the connecting map 
in the Mayer--Vietoris sequence, and let $H_w \subset H_2(G;\Z)$ be the affine subspace
mapping to the generator.
If $w$ has infinite order in $J$ and $K$, then 
$$\inf_{\alpha\in H_w} \|\alpha\| = 2(\scl_J(w) + \scl_K(w))$$
\end{proposition}
\begin{proof}
This is not difficult to see directly
from the definition, and it is very similar to the proof of
Theorem~3.4 in \cite{Calegari_surface}. However, for the sake of clarity we give an argument.
Note by the way that the hypothesis that $w$ is homologically trivial on either side
is equivalent to the statement that the inclusion map 
$H_1(\langle w \rangle;\Z) \to H_1(J;\Z)\oplus H_1(K;\Z)$ is the zero map, so $\phi$ as above
is certainly surjective. Moreover, if $H_2(J;\Z)$ and $H_2(K;\Z)$ are trivial (as will often
be the case below), then $\phi$ is an isomorphism, and $H_w$ consists of a single class
$\alpha$.

It is convenient to geometrize this algebraic picture, so let $X_J$ and $X_K$ be Eilenberg-MacLane
spaces for $J$ and $K$, and let $X_G$ be obtained from $X_J$ and $X_K$ by attaching the two ends of
a cylinder $C$ to loops representing the conjugacy classes corresponding to the images of $w$ in
either side. Let $\gamma$ be the core of $C$. If $S$ is a closed, oriented surface with no
sphere components, and $f:S \to X_G$ represents some $n\alpha$ with $\alpha \in H_w$, then
we can homotope $f$ so that it meets $\gamma$ transversely and efficiently --- i.e. so that
$f^{-1}(\gamma)$ consists of pairwise disjoint essential simple curves in $S$. If one of these
curves maps to $\gamma$ with degree zero we can compress $S$ and reduce its complexity, so without
loss of generality every component maps with nonzero degree. Hence we can cut $S$ into $S_J$ and $S_K$
each mapping to $X_J$ and $X_K$ respectively and with boundary representing some finite cover of
$w$. By definition this shows $\inf_{\alpha \in H_w} \|\alpha\| \ge 2(\scl_J(w) + \scl_K(w))$.

Conversely, given surfaces $S_J$ and $S_K$ mapping to $X_J$ and $X_K$ with boundary representing
finite covers of $w$ (or rather its image in each side), we need to construct a suitable $S$ as above.
First, we can pass to a cover of each $S_J$ and $S_K$ in such a way that the boundary of each maps to
$w$ with positive degree; see e.g. \cite{Calegari_scl}, Prop.~2.13. Then we can pass to a further
finite cover of each so that the set of degrees with which components of $\partial S_J$ and 
of $\partial S_K$ map over $w$ are the same (with multiplicity); again, see the argument of 
\cite{Calegari_scl} Prop.~2.13. Once this is done we can glue up $S_J$ to $S_K$ with the opposite
orientation to build a surface $S$ mapping to $X_G$ which, by construction, represents a multiple 
of some $\alpha$ in $H_w$. We therefore obtain $\inf_{\alpha \in H_w} \|\alpha\| \le 2(\scl_J(w) + \scl_K(w))$
and we are done.
\end{proof}

We now show how to use these Propositions to produce examples.

\begin{example}\label{first_example}
Start with a free group; for concreteness, let $F=\langle a,b,c\rangle$.
Consider a word $w\in F$ of the form $w=[a,b][c,v]$ for some $v\in F$.
Associated to this expression of $w$ as a product of two commutators 
is a genus $2$ surface $S$ with one boundary component mapping to
a $K(F,1)$ in such a way that its boundary represents $w$. This surface
is not injective, since the image of its fundamental group is $F$
which has rank $3$. Let $G=\langle a,b,c,x,y \; | \; w=[x,y]\rangle$;
i.e. geometrically a $K(G,1)$ is obtained from a $K(F,1)$ by attaching
the boundary of a once-punctured torus $T$ to $w$. The surface
$R:=S \cup T$ has genus $3$, and represents the generator of $H_2(G;\Z)$.
On the other hand, by the Amalgamation Proposition,
the Gromov--Thurston norm of this homology class
is equal to $2\cdot\scl_{\langle x,y\rangle}([x,y])+2\cdot\scl_F(w)$. Since
$\scl_{\langle x,y\rangle}([x,y])=1/2$ (see e.g.\/ \cite{Calegari_scl} Ex.~2.100),
providing $1/2 < \scl(w)$ the result
is non-injective but incompressible.

The group $G$ can be embedded in $\PSL(2,\C)$ by first embedding $F$
as a discrete subgroup, then embedding $\langle x,y\rangle$ in
such a way that $[x,y] = w$. By conjugating $\langle x,y\rangle$
by a generic loxodromic element with the same axis as $w$, we can ensure
this example is injective, and it can even be taken to be all loxodromic.
This follows in the usual way by a Bass--Serre type argument; see e.g.\/
a similar argument in \cite{Calegari_Dunfield}, Lem.~1.5.

Almost any word $v$ will give rise to $w$ with $\scl(w)>1/2$; for example,
$$\scl([a,b][c,aa])=1$$ 
as can be computed using {\tt scallop}. 
Experimentally, it appears that if $v$ is
chosen to be random of length $n$, then $\scl(w) \to 3/2$ as $n \to \infty$.
For example,
$$\scl ([a,b][c,bcABBcABCbbcACbcBcbb]) = 7/5$$
The closer $\scl(w)$ is to $3/2$, the bigger the index of a cover
in which some simple loop compresses. This gives a practical
method to produce examples for any given $k$ in which no loop with fewer
than $k$ self-crossings is in the kernel.
\end{example}

\begin{example}
Note that the groups $G$ produced in Example~\ref{first_example} are
$1$-relator groups, which are very similar to $3$-manifold groups in some
important ways. A modified construction shows they can in fact be taken
to be $1$-relator fundamental groups of hyperbolic $4$-manifolds. To see this, we consider
examples of the form $G=\langle a,b,c,x_1,y_1,\cdots,x_g,y_g \; | \; w=\prod_{i=1}^g [x_i,y_i]\rangle$
i.e.\/ we attach a once-punctured surface $T_g$ of genus $g$, giving rise to
a noninjective incompressible surface $R=S\cup T_g$ of genus $g+2$.

Let $\langle a,b,c \rangle$ act discretely and faithfully, stabilizing a totally
geodesic $\H^3$ in $\H^4$. We can arrange for the axis $\ell$ 
of $w$ to be disjoint from its translates. 
Thinking of $\langle x_1,y_1,\cdots,x_g,y_g\rangle$ as the fundamental group
of a once-punctured surface $T_g$, we choose a hyperbolic structure on this
surface for which $\partial T_g$ is isometric to $\ell/\langle w\rangle$,
and make this group act by stabilizing a totally geodesic $\H^2$ in $\H^4$
in such a way that the axis of $\partial T_g$
intersects the $\H^3$ perpendicularly along $\ell$. Providing the
shortest essential arc in $T_g$ from $\partial T_g$ to itself is sufficiently
long (depending on the minimal 
distance from $\ell$ to its translates by $\langle a,b,c\rangle$)
the resulting group is discrete and faithful. This
follows by applying the Klein--Maskit combination theorem, once we ensure that the
limit sets of the conjugates of $\langle a,b,c\rangle$ are contained in
regions satisfying the ping-pong hypothesis for the action of $\pi_1(T_g)$. 
This condition can be ensured by taking $g$
big enough and choosing the hyperbolic structure on $T_g$ accordingly; the details
are entirely straightforward.
\end{example}

\begin{example}
Let $H$ be any nonelementary
hyperbolic $2$-generator group which is torsion free but not free.
Let $a,b$ be the generators. Then the once-punctured torus with boundary
$[a,b]$ is not injective. As before, let
$G = \langle H,x,y\; | \; [a,b]=[x,y]\rangle$. Then $G$ contains a genus
$2$ surface representing the amalgamated class in $H_2(G;\Z)$, and the
norm of this class is $1+2\cdot \scl_H([a,b]) > 0$, so this example
is noninjective but incompressible.

As an example, we could take $H$ to be the fundamental group of
a closed hyperbolic $3$-manifold of Heegaard genus $2$, 
or a $2$-bridge knot complement. Such
examples have discrete faithful representations into $\PSL(2,\C)$.
\end{example}


\begin{example}
It is easy to produce examples of $2$-generator 1-relator groups 
$H=\langle a,b \; | \; v \rangle$ in which $1/2 - \epsilon < \scl([a,b]) < 1/2$
for any $\epsilon$. Such groups are torsion-free if $v$ is not a proper
power. Just fix some big integer $N$ and 
take $v$ to be any product of conjugates
$$v = ([a,b]^{\pm N})^{g_1}([a,b]^{\pm N})^{g_2} \cdots ([a,b]^{\pm N})^{g_m}$$
for which there are as many $+N$'s as $-N$'s. Such an $H$
maps to the Seifert-fibered $3$-manifold
group $\langle a,b,z \; | \; [a,b]^N = z^{N-1}, [a,z]=[b,z]=1\rangle$
in which $\scl([a,b])=(N-1)/2N$. The only subtle part of this last equality is the lower
bound, which is certified by Bavard duality (see \cite{Calegari_scl} Thm.~2.70)
and the existence of a rotation quasimorphism associated to a realization of the fundamental
group of the Seifert manifold as a central extension of 
the fundamental group of a hyperbolic torus orbifold with
one orbifold point of order $N$. Since $\scl$ is monotone nonincreasing
under homomorphisms, the claim follows.
\end{example}

I would like to thank Lars Louder, Jason Manning and the anonymous referee for helpful conversations
and suggestions. I would also like to thank Jason for suggesting that I call this 
paper ``scl, SLC and $SL(2,\C)$'' and for understanding why I decided not to. Danny Calegari 
was partly supported by NSF grant DMS 1005246.


\begin{thebibliography}{99}
\bibitem{Bowditch}
  B. Bowditch,
  \emph{Markoff triples and quasi-Fuchsian groups},
  Proc. London Math. Soc. (3) {\bf 77} (1998), no. 3, 697--736
\bibitem{Calegari_scl}
  D. Calegari,
  \emph{scl},
  MSJ Memoirs, {\bf 20}. Mathematical Society of Japan, Tokyo, 2009.
\bibitem{Calegari_surface}
  D. Calegari,
  \emph{Surface subgroups from homology},
  Geom. Topol. {\bf 12} (2008), 1995--2007
\bibitem{Calegari_Dunfield}
  D. Calegari and N. Dunfield,
  \emph{An ascending HNN extension of a free group inside $\SL(2,\C)$},
  Proc. AMS {\bf 134} (2006), no. 11, 3131--3136
\bibitem{Calegari_Fujiwara}
  D. Calegari and K. Fujiwara,
  \emph{Stable commutator length in word hyperbolic groups},
  Groups, Geom. Dyn. {\bf 4} (2010), no. 1, 59--90
\bibitem{scallop}
  D. Calegari and A. Walker,
  {\tt scallop}, 
  computer program available from the authors' webpages and from
  {\tt computop.org}
\bibitem{Cooper_Manning}
  D. Cooper and J. Manning,
  \emph{Non-faithful representations of surface groups into 
  $SL(2,\C)$ which kill no simple closed curve},
  preprint, arXiv:1104.4492
\bibitem{Gabai}
  D. Gabai,
  \emph{Foliations and the topology of $3$-manifolds},
  J. Diff. Geom. {\bf 18} (1983), no. 3, 445--503
\bibitem{Gromov}
  M. Gromov,
  \emph{Volume and bounded cohomology},
  IHES Pub. Math. (1982), no. 56, 5--99
\bibitem{Louder}
  L. Louder,
  \emph{Simple loop conjecture for limit groups},
  preprint, arXiv:1106.1350
\bibitem{Minsky}
  Y. Minsky,
  \emph{Short geodesics and end invariants},
  Comprehensive research on complex dynamical systems and related fields 
  (Japanese) (Kyoto, 1999). 
	S\=urikaisekikenky\=usho K\=oky\=uroku No. 1153 (2000), 1–19
\bibitem{Thurston}
  W. Thurston,
  \emph{A norm on the homology of $3$-manifolds},
  Mem. AMS. {\bf 59} (1986), no. 339, i--vi and 99--130
\end{thebibliography}
\end{document}